\documentclass{article}

\usepackage[english]{babel}
\usepackage{amsmath, amssymb, amsthm, amsfonts, bm}
\usepackage{enumerate}
\usepackage{graphicx, color}
\usepackage[colorlinks=true]{hyperref}
\usepackage{enumitem}

\numberwithin{equation}{section}

\newtheorem{lemma}{Lemma}[section]
\theoremstyle{definition}

\theoremstyle{remark}

\newcommand{\Rnum}[1]{\uppercase\expandafter{\romannumeral #1\relax}}

\newcommand{\mr}[1]{\mathrm{#1}}
\newcommand{\mb}[1]{\mathbb{#1}}
\newcommand{\mc}[1]{\mathcal{#1}}
\newcommand{\vsal}{\\\noalign{\vspace{0.2cm}}}

\DeclareMathOperator{\Av}{Av} 

\def\clap#1{\hbox to 0pt{\hss#1\hss}}

\title{On a constant related to the Bellman function of three integral variables of the dyadic maximal operator: Part A}
\author{Eleftherios N. Nikolidakis}
\date{\today}

\begin{document}
\maketitle

\begin{abstract}
We study the behaviour of the constant that is provided in the articles \cite{12} and \cite{13}, which is connected with the determination of the Bellman function of three integral variables of the dyadic maximal operator. More precisely we study the monotonicity properties of this constant with respect to the second variable from which it depends. 
\end{abstract}

\section{Introduction} \label{sec:0}
The dyadic maximal operator on $\mb R^n$ is a useful tool in analysis and is defined by
\begin{equation} \label{eq:0p1}
	\mc M_d\varphi(x) = \sup\left\{ \frac{1}{|S|} \int_S |\varphi(u)|\,\mr du: x\in S,\ S\subseteq \mb R^n\ \text{is a dyadic cube} \right\},
\end{equation}
for every $\varphi\in L^1_\text{loc}(\mb R^n)$, where $|\cdot|$ denotes the Lebesgue measure on $\mb R^n$, and the dyadic cubes are those formed by the grids $2^{-N}\mb Z^n$, for $N=0, 1, 2, \ldots$.\\
It is well known that it satisfies the following weak type (1,1) inequality
\begin{equation} \label{eq:0p2}
	\left|\left\{ x\in\mb R^n: \mc M_d\varphi(x) > \lambda \right\}\right| \leq \frac{1}{\lambda} \int_{\left\{\mc M_d\varphi > \lambda\right\}} |\varphi(u)|\,\mr du,
\end{equation}
for every $\varphi\in L^1(\mb R^n)$, and every $\lambda>0$,
from which it is easy to get the following  $L^p$-inequality
\begin{equation} \label{eq:0p3}
	\|\mc M_d\varphi\|_p \leq \frac{p}{p-1} \|\varphi\|_p,
\end{equation}
for every $p>1$, and every $\varphi\in L^p(\mb R^n)$.
It is easy to see that the weak type inequality \eqref{eq:0p2} is the best possible. For refinements of this inequality one can consult \cite{6}.

It has also been proved that \eqref{eq:0p3} is best possible (see \cite{1} and \cite{2} for general martingales and \cite{21} for dyadic ones).
An approach for studying the behaviour of this maximal operator in more depth is the introduction of the so-called Bellman functions which play the role of generalized norms of $\mc M_d$. Such functions related to the $L^p$-inequality \eqref{eq:0p3} have been precisely identified in \cite{4}, \cite{5} and \cite{14}. For the study of the Bellman functions of $\mc M_d$, we use the notation $\Av_E(\psi)=\frac{1}{|E|} \int_E \psi$, whenever $E$ is a Lebesgue measurable subset of $\mb R^n$ of positive measure and $\psi$ is a real valued measurable function defined on $E$. We fix a dyadic cube  $Q$ and define the localized maximal operator $\mc M'_d\varphi$ as in \eqref{eq:0p1} but with the dyadic cubes $S$ being assumed to be contained in $Q$. Then for every $p>1$ we let
\begin{equation} \label{eq:0p4}
	B_p(f,F)=\sup\left\{ \frac{1}{|Q|} \int_Q (\mc M'_d\varphi)^p: \Av_Q(\varphi)=f,\ \Av_Q(\varphi^p)=F \right\},
\end{equation}
where $\varphi$ is nonnegative in $L^p(Q)$ and the variables $f, F$ satisfy $0<f^p\leq F$. By a scaling argument it is easy to see that \eqref{eq:0p4} is independent of the choice of $Q$ (so we may choose
$Q$ to be the unit cube $[0,1]^n$).
In \cite{5}, the function \eqref{eq:0p4} has been precisely identified for the first time. The proof has been given in a much more general setting of tree-like structures on probability spaces.

More precisely we consider a non-atomic probability space $(X,\mu)$ and let $\mc T$ be a family of measurable subsets of $X$, that has a tree-like structure similar to the one in the dyadic case (the exact definition can be seen in \cite{5}).
Then we define the dyadic maximal operator associated to $\mc T$, by
\begin{equation} \label{eq:0p5}
	\mc M_{\mc T}\varphi(x) = \sup \left\{ \frac{1}{\mu(I)} \int_I |\varphi|\,\mr \; d\mu: x\in I\in \mc T \right\},
\end{equation}
for every $\varphi\in L^1(X,\mu)$, $x\in X$.

This operator is related to the theory of martingales and satisfies essentially the same inequalities as $\mc M_d$ does. Now we define the corresponding Bellman function of four variables of $\mc M_{\mc T}$, by
\begin{multline} \label{eq:0p6}
	B_p^{\mc T}(f,F,L,k) = \sup \left\{ \int_K \left[ \max(\mc M_{\mc T}\varphi, L)\right]^p\mr \; d\mu: \varphi\geq 0, \int_X\varphi\,\mr \; d\mu=f, \right. \\  \left. \int_X\varphi^p\,\mr \; d\mu = F,\ K\subseteq X\ \text{measurable with}\ \mu(K)=k\right\},
\end{multline}
the variables $f, F, L, k$ satisfying $0<f^p\leq F $, $L\geq f$, $k\in (0,1]$.
The exact evaluation of \eqref{eq:0p6} is given in \cite{5}, for the cases where $k=1$ or $L=f$. In the first case, the author (in \cite{5}) precisely identifies the function $B_p^{\mc T}(f,F,L,1)$ by evaluating it in a first stage for the case where $L=f$. That is he precisely identifies $B_p^{\mc T}(f,F,f,1)$ (in fact $B_p^{\mc T}(f,F,f,1)=F \omega_p (\frac{f^p}{F})^p$, where                         $\omega_p: [0,1] \to [1,\frac{p}{p-1}]$ is the inverse function $H^{-1}_p$, of $H_p(z) = -(p-1)z^p + pz^{p-1}$). 

The proof of the above mentioned evaluation relies on a one-parameter integral inequality which is proved by arguments based on a linearization of the dyadic maximal operator. More precisely the author in \cite{5} proves that the inequality

\begin{equation}\label{eq:0p7}
	F\geq \frac{1}{(\beta+1)^{p-1}} f^p + \frac{(p-1)\beta}{(\beta+1)^p} \int_X (M_{\mathcal{T}}\varphi)^p \; d\mu,
\end{equation}
is true for every non-negative value of the parameter $\beta$ and sharp for one that depends on $f$, $F$ and $p$, namely for $\beta=\omega_p (\frac{f^p}{F})-1$. This gives as a consequence an upper bound for $B_p^{\mc T}(f,F,f,1)$, which after several technical considerations is proved to be best possible.Then  by using several calculus arguments the author in \cite{5} provides the evaluation of $B_p^{\mc T}(f,F,L,1)$ for every $L\geq f$. 

Now in \cite{14} the authors give a direct proof of the evaluation of $B_p^{\mc T}(f,F,L,1)$ by using alternative methods. Moreover in the second case, where $L=f$, the author (in \cite{5}) uses the evaluation of $B_p^{\mc T}(f,F,f,1)$ and provides the evaluation of the more general $B_p^{\mc T}(f,F,f,k)$, $k\in (0,1]$.

Our aim in this article is to study further the results of \cite{12} and \cite{13} in order to approach the following Bellman function problem (of three integral variables)

\begin{multline} \label{eq:0p8}
	B_{p,q}^{\mc T}(f,A,F) = \sup \left\{ \int_X \left(\mc M_{\mc T}\varphi\right)^p\mr \; d\mu: \varphi\geq 0, \int_X\varphi\,\mr \; d\mu=f, \right. \\  \left. \int_X\varphi^q\,\mr \; d\mu = A,\ \int_X\varphi^p\,\mr \; d\mu = F\right\},
\end{multline}
where $1<q<p$, and the variables $f,A,F$ lie in the domain of definition of the above problem. Certain progress for the above problem can be seen in \cite{11}.

In \cite{12} it is proved that whenever $0<\frac{x^q}{\kappa^{q-1}}<y\leq x^{\frac{p-q}{p-1}}\cdot z^{\frac{q-1}{p-1}}\;\Leftrightarrow\; 0<s_1^{\frac{q-1}{p-1}}\leq s_2<1$, (where $s_1,s_2$ are defined right below), there is a constant $t=t(s_1,s_2)$ for which if $h:(0,\kappa]\longrightarrow\mb{R}^{+}$ satisfies  $\int_{0}^{\kappa}h=x$\,,\; $\int_{0}^{\kappa}h^q=y$ and $\int_{0}^{\kappa}h^p=z$ then

	$$\int_{0}^{\kappa}\bigg(\frac{1}{t}\int_{0}^{t}h\bigg)^p dt\leq t^p(s_1,s_2)\cdot\int_{0}^{\kappa}h^p.$$

More precisely $t(s_1,s_2)=t$ is the greatest element of $\big[{1,t(0)}\big]$ for which $F_{s_1,s_2}(t)\leq 0$ where $F_{s_1,s_2}$ is defined in \cite{12}. Moreover for each such fixed $s_1,s_2$ 
\[t=t(s_1,s_2)=\min\Big\{{t(\beta)\;:\;\beta\in\big[0,\tfrac{1}{p-1}\big]}\Big\}\]
where $t(\beta)=t(\beta,s_1,s_2)$ is defined in \cite{12}. That is we find a constant $t=t(s_1,s_2)$ for which the above inequality  is satisfied for all $h:(0,\kappa]\longrightarrow\mb{R}^{+}$ as mentioned above. Note that $s_1,s_2$ depend by a certain way on $x,y,z$, namely $s_1=\frac{x^p}{\kappa^{p-1}z}$,  $s_2=\frac{x^q}{\kappa^{q-1}y}$ and $F_{s_1,s_2}(\cdot)$ is given in terms of $s_1,s_2$.

In this article we study the monotonicity behaviour of the function 
$t(s_1,s_2)$ with respect to the second variable $s_2$. This may enable us in the future to determine \eqref{eq:0p8}, by using also the results in \cite{11}, \cite{12} and \cite{13}.

We need to mention that the extremizers for the standard Bellman function $B_p^{\mc T}(f,F,f,1)$ have been studied in \cite{7}, and in \cite{9} for the case $0<p<1$. Also in \cite{8} the extremal sequences of functions for the respective Hardy operator problem have been studied.  Additionally further study of the dyadic maximal operator can be seen in \cite{10,14} where symmetrization principles for this operator are presented, while other approaches for the determination of certain Bellman functions are given in \cite{16,17,18,19,20}. Moreover results related to applications of dyadic maximal operators can be seen in \cite{15}.

\bigskip

\section{The monotonicity properties of $t(s_1,s_2)$ with respect to the variable $s_2$} \label{sec:2}

As we already seen in \cite{12}, we have defined a function $t(s_1,s_2)$ for all $({s_1,s_2})$ satisfying $0<s_1^{q-1}\leq s_2^{p-1}<1$ by the following manner:
\begin{enumerate}
\item [a)]
If $t'_{s_1,s_2}(0)\leqslant0$, then $t(s_1,s_2)=t$ is the unique solution in the interval $\big[{1,t(0)}\big]$, of the equation
\[q\,\big({p\,\omega_q(\tau)^{q-1}-(p-1)\,\omega_q(\tau)^q}\big)\Big({t^{p-q}-\frac{s_1}{s_2}}\Big)=(p-q)\,s_1\,\alpha(s_2)\,,\]
where 
\[\alpha(s_2):=\frac{\omega_q(s_2)^q}{s_2}-1\,,\quad \tau({s_1,s_2,t})=\tau=\frac{p-q}{p}\,\frac{t^p-s_1}{t^{p-q}-\frac{s_1}{s_2}}\,,\]
where also $t(0)=t_{s_1,s_2}(0)$ satisfies:
\[t^{p}({0})-\frac{p}{p-q}\,t^{p-q}(0)=s_1-\frac{p}{p-q}\,\frac{s_1}{s_2}=:h(s_1,s_2)\,.\]
\item [b)]
If $t'_{s_1,s_2}(0)>0$, then $t(s_1,s_2)=t$ is given by $t=t(0)$ as described right above.
\end{enumerate}

It is obvious that $t(s_1,s_2)$ that is defined on 
\[D=\Big\{{(s_1,s_2)\in \mb{R}^2\;:\; \text{with}\;0<s_1^{q-1}\leq s_2^{p-1}<1 }\Big\}\]
is continuous. Note  also that $t'_{s_1,s_2}(0)<0\;\Leftrightarrow\; h^{-1}\big({\frac{1}{s_1^q}}\big)<s_2<1$, for every $(s_1,s_2)\in D$, that is the case a) above holds iff $h^{-1}\big({\frac{1}{s_1^q}}\big)<s_2<1$ where $h(s_2)$ is defined in \cite{13}.

Additionally each point  $(s_1,s_2)\in D$ arises by considering a triple of points $(x,y,z)$ which satisfy the relations $\frac{x^q}{\kappa^{q-1}}<y\leq x^{\frac{p-q}{p-1}}\cdot  z^{\frac{q-1}{p-1}}$, and then by defining  $s_1=\frac{x^p}{\kappa^{p-1}z}$, $s_2=\frac{x^q}{\kappa^{q-1}y}$. 

As we have seen in \cite{13}, when $(s_1,s_2)\in D$ satisfies $t'_{s_1,s_2}(0)>0\;\Leftrightarrow\; s_1^{\frac{q-1}{p-1}}\leq s_2<h^{-1}\big({\frac{1}{s_1^q}}\big)$, then $t(s_1,s_2)=t_{s_1,s_2}(0)=t(0)$ satisfies $\frac{\partial t(0)}{\partial s_2}>0$, that is whenever $s_1\in({0,1})$ is such that $s_1^{\frac{q-1}{p-1}}<h^{-1}\big({\frac{1}{s_1^q}}\big)$ the function $t(s_1,s_2)$ is strictly increasing on the interval $\Big[{s_1^{\frac{q-1}{p-1}},h^{-1}\big({\frac{1}{s_1^q}}\big)}\Big]$. We continue in this way, by studying the behavior of $t(s_1,s_2)$, for each fixed $s_1\in (0,1)$, where $s_2\in \Big[{\max\Big({s_1^{\frac{q-1}{p-1}},h^{-1}\big({\frac{1}{s_1^q}}\big)}\Big),1}\Big]$, where $t({s_1,s_2})$ is given by the description right above (see case a)\,).

More precisely we shall prove that 
\begin{align}
\left.{\begin{array}{ll}
\dfrac{\partial t}{\partial s_2}>0\,,& \forall\,s_2\in \Big[{\max\Big({s_1^{\frac{q-1}{p-1}},h^{-1}\big({\frac{1}{s_1^q}}\big)}\Big),s'_2}\Big)\vsal
\dfrac{\partial t}{\partial s_2}<0\,,&\forall\,s_2\in({s'_2,1})
\end{array}}\right\}\label{eq:2p1}
\end{align}
where $s'_2$ is defined by $s'_2=H_q\big({\omega_p(s_1)}\big)$. 

Note that  for this choice of $s'_2$ we have that $s'_2>s_1^{\frac{q-1}{p-1}}$. Indeed this last mentioned inequality is equivalent to $H_q\big({\omega_p(s_1)}\big)>s_1^{\frac{q-1}{p-1}}\;\Leftrightarrow\; \omega_q\big({s_1^{\frac{q-1}{p-1}}}\big)>\omega_p(s_1)$,\; $s_1\in ({0,1})$, which is true in view of the results in \cite{13} (see Lemma 2.6).

Moreover $s'_2=H_q\big({\omega_p(s_1)}\big)$, satisfies $s'_2>h^{-1}\big({\frac{1}{s_1^q}}\big)$, thus $t_{s_1,s'_2}(0)<0$ (see \cite{13}). 

Note also that \eqref{eq:2p1} means exactly that 
\begin{align}
\left.{\begin{array}{ll}
\dfrac{\partial t}{\partial y}<0\,,& \forall\,y\in \bigg({y_0,y':=\min\bigg\{{x^{\frac{p-q}{p-1}} z^{\frac{q-1}{p-1}},\frac{x^q}{\kappa^{q-1}h^{-1}\big({\frac{1}{s_1^q}}\big)}}\bigg\}}\bigg)\vsal
\dfrac{\partial t}{\partial y}>0\,,&\forall\,y\in\big({\frac{x^q}{\kappa^{q-1}},y_0}\big)
\end{array}}\right\}\label{eq:2p2}
\end{align}
where $y_0$ is defined by the equation $y_0=\frac{x^q}{\kappa^{q-1}H_q\big({\omega_p\big({\frac{x^p}{\kappa^{p-1}z}}\big)}\big)}$. This is true in view of the observation which states that if we define  $\lambda({x,y,z})=t({s_1,s_2})$, with $s_1=\frac{x^p}{\kappa^{p-1}z}$, and $s_2=s_2(y)=\frac{x^q}{\kappa^{q-1}y}$ then we obviously have that 
\begin{align*}
\frac{\partial \lambda}{\partial y}=\frac{\partial t}{\partial s_2}\cdot s'_2(y)=\frac{\partial t}{\partial s_2}\cdot \Big({-\frac{x^q}{\kappa^{q-1}y^2}}\Big)\,.
\end{align*}
Note that 
\begin{align}
y_0=\frac{x^q}{\kappa^{q-1}H_q\big({\omega_p\big({\frac{x^p}{\kappa^{p-1}z}}\big)}\big)}\label{eq:2p3}
\end{align}
satisfies 
\begin{align}
\frac{x^q}{\kappa^{q-1}}<y_0\leq x^{\frac{p-q}{p-1}}\cdot z^{\frac{q-1}{p-1}}\,.\label{eq:2p4}
\end{align}
\eqref{eq:2p4} is obviously true since $s_1=\frac{x^p}{\kappa^{p-1}z}<1$ and $s'_2=H_q\big({\omega_p({s_1})}\big)>s_1^{\frac{q-1}{p-1}}$ (see Lemma 2.6 in \cite{13}). At this point we mention that the equation $s'_2=H_q\big({\omega_p({s_1})}\big)$, where  $s'_2=\frac{x^q}{\kappa^{q-1}y}$ for some $y$, states that $y$ must be equal to $y_0$, which is defined by \eqref{eq:2p3}. We remind also that if $\alpha({s_2})=\frac{\omega_q({s_2})^q}{s_2}-1$, then (see Lemma 3.2 in \cite{13}) $\alpha'({s_2})$ is given by 
\begin{align*}
\alpha'({s_2})=\frac{1}{s_2}\bigg[{-\frac{\omega_q(s_2)}{(q-1)({\omega_q}(s_2)-1)}-\frac{\omega_q({s_2})^q}{s_2}}\bigg],\quad s_2\in({0,1})\,.
\end{align*}

In order to prove \eqref{eq:2p1} we must use the definition of $t(s_1,s_2)$ which is given for $({s_1,s_2})\in D$ for which $\max\big({s_1^{\frac{q-1}{p-1}},h^{-1}\big({\frac{1}{s_1^q}}\big)}\big)<s_2<1$. For those $({s_1,s_2})$, and from the results of \cite{13}, the following identity should be true: 

$$\Gamma({s_1,s_2})=\Delta({s_1,s_2}),$$ where $\Gamma$, $\Delta$ are defined as in the beginning of this Section: 
\begin{align*}
\Gamma({s_1,s_2})&:=q\big({p\,\omega_q(\tau)^{q-1}-({p-1})\,\omega_q({\tau})^q}\big)\Big({t^{p-q}-\frac{s_1}{s_2}}\Big)\,,\vsal
\Delta({s_1,s_2})&:=({p-q})\,s_1\cdot\alpha(s_2)\,.
\end{align*}
We now evaluate $\frac{\partial \Gamma}{\partial s_2}$. We have 
\begin{align}
\frac{\partial \Gamma}{\partial s_2}&=q\bigg[{p({q-1})\,\omega_q(\tau)^{q-2}\frac{\partial \,\omega_q(\tau)}{\partial s_2}-({p-1})q\,\omega_q(\tau)^{q-1}\frac{\partial \,\omega_q(\tau)}{\partial s_2}}\bigg]\Big({t^{p-q}-\frac{s_1}{s_2}}\Big)\nonumber\vsal
 &\qquad +q\,\Big({p\,\omega_q(\tau)^{q-1}-({p-1})\,\omega_q(\tau)^{q}}\Big)\bigg[{({p-q})\,t^{p-q-1}\frac{\partial t}{\partial s_2}+\frac{s_1}{s_2^2}}\bigg]\,.\label{eq:2p5}
\end{align}
We now evaluate $\frac{\partial \,\omega_q(\tau)}{\partial s_2}$. For this purpose we will need the partial derivative 
\begin{align}
\frac{\partial\tau}{\partial s_2}&=\frac{\partial}{\partial s_2}\bigg({\frac{p-q}{p}\,\frac{t^p-s_1}{t^{p-q}-\frac{s_1}{s_2}}}\bigg)
\nonumber\vsal
 &=\frac{p-q}{p}\,({t^p-s_1})\,\frac{1}{\big({t^{p-q}-\frac{s_1}{s_2}}\big)^2}\,({-1})\,\bigg({({p-q})\,t^{p-q-1}\frac{\partial\,t}{\partial s_2}+\frac{s_1}{s_2^2}}\bigg)\,+\nonumber\vsal
&\hspace{5.50cm}+\frac{p-q}{p}\,p\,t^{p-1}\frac{\partial\,t}{\partial s_2}\,\frac{1}{t^{p-q}-\frac{s_1}{s_2}}\,.\label{eq:2p6}
\end{align}
Then also  
\begin{align}
\frac{\partial }{\partial s_2}\,\omega_q(\tau)=\omega'_q(\tau)\cdot\frac{\partial\,\tau}{\partial s_2}=\frac{1}{q({q-1})\,\omega_q(\tau)^{q-2}\big({1-\,\omega_q(\tau)}\big)}\cdot\frac{\partial\,\tau}{\partial s_2}\,.\label{eq:2p7}
\end{align}
By \eqref{eq:2p5}, \eqref{eq:2p6} and \eqref{eq:2p7} we get:

\begin{align*}
\frac{\partial \Gamma}{\partial s_2}&=\bigg[q\,p\,(q-1)\,\omega_q(\tau)^{q-2}\frac{1}{q({q-1})\,\omega_q(\tau)^{q-2}\big({1-\omega_q(\tau)}\big)}\cdot\frac{\partial\,\tau}{\partial s_2}\,-\vsal
 &-\,q^2\,(p-1)\,\omega_q(\tau)^{q-1}\frac{1}{q({q-1})\,\omega_q(\tau)^{q-2}\big({1-\,\omega_q(\tau)\big)}}\cdot\frac{\partial\,\tau}{\partial s_2}\bigg]\Big({t^{p-q}-\frac{s_1}{s_2}}\Big)\,+\vsal
 &\hspace{2.0cm}+q\Big({p\,\omega_q(\tau)^{q-1}-({p-1})\,\omega_q(\tau)^{q}}\Big)\bigg({({p-q})\,t^{p-q-1}\frac{\partial\,t}{\partial s_2}+\frac{s_1}{s_2^2}}\bigg)\vsal
 &=\bigg({\frac{\partial\,\tau}{\partial s_2}\,\frac{p}{1-\omega_q(\tau)}-\frac{\partial\,\tau}{\partial s_2}\,\frac{q(p-1)}{q-1}\,\frac{\omega_q(\tau)}{1-\omega_q(\tau)}}\bigg)\Big({t^{p-q}-\frac{s_1}{s_2}}\Big)\,+\vsal
 &\hspace{2.0cm} +q\Big({p\,\omega_q(\tau)^{q-1}-({p-1})\,\omega_q(\tau)^{q}}\Big)\bigg({({p-q})\,t^{p-q-1}\frac{\partial\,t}{\partial s_2}+\frac{s_1}{s_2^2}}\bigg)
  \end{align*}

\begin{align*}
 &=\bigg({\frac{p}{1-\omega_q(\tau)}-\frac{q(p-1)}{q-1}\,\frac{\omega_q(\tau)}{1-\omega_q(\tau)}}\bigg)\Big({t^{p-q}-\frac{s_1}{s_2}}\Big)\Bigg\{\frac{p-q}{p}\,({t^p-s_1})\,\cdot\vsal
 &\qquad \frac{1}{\big({t^{p-q}-\frac{s_1}{s_2}}\big)^2}\,({-1})\bigg({(p-q)\,t^{p-q-1}\frac{\partial\,t}{\partial s_2}+\frac{s_1}{s_2^2}}\bigg)+({p-q})\,t^{p-1}\frac{\partial\,t}{\partial s_2}\,\frac{1}{t^{p-q}-\frac{s_1}{s_2}} \Bigg\}\vsal
 &\hspace{2.0cm} +q\Big({p\,\omega_q(\tau)^{q-1}-({p-1})\,\omega_q(\tau)^{q}}\Big)\bigg({({p-q})\,t^{p-q-1}\frac{\partial\,t}{\partial s_2}+\frac{s_1}{s_2^2}}\bigg)\vsal
 &=\bigg({\frac{p}{1-\omega_q(\tau)}-\frac{q(p-1)}{q-1}\,\frac{\omega_q(\tau)}{1-\omega_q(\tau)}}\bigg)\Bigg[({-\tau})\bigg({({p-q})\,t^{p-q-1}\frac{\partial\,t}{\partial s_2}+\frac{s_1}{s_2^2}}\bigg)\,+\vsal
 &\qquad+ ({p-q})\,t^{p-1}\frac{\partial\,t}{\partial s_2}\Bigg]+q\Big({p\,\omega_q(\tau)^{q-1}-({p-1})\,\omega_q(\tau)^{q}}\Big)\,\cdot\vsal
 &\hspace{7.0cm}\bigg({({p-q})\,t^{p-q-1}\frac{\partial\,t}{\partial s_2}+\frac{s_1}{s_2^2}}\bigg)\quad\Rightarrow
\end{align*}
\begin{align}
\frac{\partial \Gamma}{\partial s_2}&= \frac{\partial t}{\partial s_2}\Bigg[-({p-q})\,t^{p-q-1}\cdot F({\tau})+(p-q)\,t^{p-1}\frac{1}{1-\omega_q(\tau)}\,\cdot\nonumber\vsal
&\hspace{4.0cm}\bigg({p-\frac{(p-1)\,q}{q-1}\,\omega_q(\tau)}\bigg)\Bigg]-\frac{s_1}{s_2^2}\cdot F(\tau)\,,\label{eq:2p8}
\end{align}
where $F(\tau)$ is defined by 
\begin{align}
F({\tau})&=\frac{\tau}{1-\omega_q(\tau)}\bigg({p-\frac{q\,(p-1)}{q-1}\,\omega_q(\tau)}\bigg)\,-\nonumber\vsal
&\hspace{4.5cm}-q\Big({p\,\omega_q(\tau)^{q-1}-({p-1})\,\omega_q(\tau)^{q}}\Big)\,.\label{eq:2p9}
\end{align}

Thus since $\Gamma(s_1,s_2)=\Delta(s_1,s_2)$, we get by differentiating on $s_2$, and using \eqref{eq:2p8}, that the following identity is true:
\begin{align}
\frac{\partial\,t}{\partial s_2}\,({p-q})\,t^{p-q-1}\,\Delta_1&=\frac{s_1}{s_2^2}\, F(\tau)\,-\nonumber\vsal
&-({p-q})\,\frac{s_1}{s_2}\bigg({\frac{\omega_q(s_2)}{(q-1)\big({\omega_q(s_2)-1}\big)}+\frac{\omega_q(s_2)^q}{s_2}}\bigg)\,,\label{eq:2p10}
\end{align}
where $\Delta_1=\Delta_1(s_1,s_2)$ is defined by 

\begin{align}
\Delta_1&:=-F({\tau})+t^q\frac{1}{1-\omega_q(\tau)}\bigg({p-\frac{q\,(p-1)}{q-1}\,\omega_q(\tau)}\bigg)\nonumber\vsal
&\stackrel{\text{by}\,\eqref{eq:2p9}}{=\!=\!=\!=}\frac{-\tau}{1-\omega_q(\tau)}\bigg({p-\frac{q\,(p-1)}{q-1}\,\omega_q(\tau)}\bigg)+q\Big({p\,\omega_q(\tau)^{q-1}-({p-1})\,\omega_q(\tau)^{q}}\Big)+\nonumber\vsal
&\hspace{5.0cm}+ t^{q}\,\frac{1}{1-\omega_q(\tau)}\bigg({p-\frac{q\,(p-1)}{q-1}\,\omega_q(\tau)}\bigg)\quad\Rightarrow\nonumber\vsal
\Delta_1&:=q\big({p-(p-1)\omega_q(\tau)}\big)\omega_q(\tau)^{q-1}+\bigg({p-\frac{q\,(p-1)}{q-1}\omega_q(\tau)}\bigg)\frac{\tau-t^q}{\omega_q(\tau)-1}.\label{eq:2p11}
\end{align}

We prove that the quantity $\Delta_1$ satisfies: $\Delta_1>0$. For this purpose we evaluate 
\begin{align}
\tau-t^q&=\frac{p-q}{p}\,\frac{t^p-s_1}{t^{p-q}-\frac{s_1}{s_2}}-t^q\nonumber\vsal
        &=\frac{(p-q)({t^p-s_1})}{p\big({t^{p-q}-\frac{s_1}{s_2}}\big)}-t^q\nonumber\vsal
        &=\frac{(p-q)({t^p-s_1})-p\,t^q\big({t^{p-q}-\frac{s_1}{s_2}}\big)}{p\big({t^{p-q}-\frac{s_1}{s_2}}\big)}\nonumber\vsal
&=\frac{p\,t^p-q\,t^p-(p-q)s_1-p\,t^{p}+p\,t^q\frac{s_1}{s_2}}{p\big({t^{p-q}-\frac{s_1}{s_2}}\big)}\quad\Rightarrow\nonumber\vsal
\tau-t^q&=\frac{p\,\frac{s_1}{s_2}\,t^q- q\,t^p-(p-q)s_1}{p\big({t^{p-q}-\frac{s_1}{s_2}}\big)}\,.\label{eq:2p12}\end{align}
Thus by \eqref{eq:2p11} and \eqref{eq:2p12} we obtain (by using also the definition of $t=t(s_1,s_2)$ in the domain of interest)
\begin{align}
\Delta_1&=\frac{({p-q})s_1\,\alpha({s_2})}{t^{p-q}-\frac{s_1}{s_2}}+\frac{1}{\omega_q(\tau)-1}\bigg({p-\frac{q\,(p-1)}{q-1}\,\omega_q(\tau)}\bigg)\,\cdot\nonumber\vsal
&\hspace{6.0cm}\frac{p\,\frac{s_1}{s_2}\,t^q- q\,t^p-(p-q)s_1}{p\,(t^{p-q}-\frac{s_1}{s_2})}\quad\Rightarrow\nonumber\vsal
\Delta_1&\cong ({p-q})s_1\,\alpha({s_2})-\bigg({\frac{q\,(p-1)}{q-1}\,\omega_q(\tau)-p}\bigg)\frac{1}{\omega_q(\tau)-1}\,E_1(s_1,s_2)\,,\label{eq:2p13}
\end{align}
where $E_1(s_1,s_2)$ is defined by 
\begin{align}
E_1(s_1,s_2)=p\,\frac{s_1}{s_2}\,t^q- q\,t^p-(p-q)s_1\,.\label{eq:2p14}
\end{align}
Moreover $\frac{q\,(p-1)}{q-1}\,\omega_q(\tau)-p>\frac{q\,(p-1)}{q-1}-p=\frac{pq-q-pq+p}{q-1}=\frac{p-q}{q-1}>0$. We now show that $E_1(s_1,s_2)<0$, $\forall\,(s_1,s_2)\in D$, implying that $\Delta_1(s_1,s_2)>0$, $\forall\,(s_1,s_2)\in D$.  We just need to show that 
\begin{align*}
- q\,t^p+p\,\frac{s_1}{s_2}\,t^q-(p-q)s_1<0\,,\quad \forall\,(s_1,s_2)\in D\,.
\end{align*}
Obviously $1\leq t=t(s_1,s_2)\leq \omega_p(s_1)<\frac{p}{p-1}$. Define the function $\varphi(u)=- q\,u^p+p\,\frac{s_1}{s_2}\,u^q-(p-q)s_1$\,, \; $u\in \big[{1,\frac{p}{p-1}}\big)$. Then $\varphi'(u)=- q\,p\,u^{p-1}+p\,q\,\frac{s_1}{s_2}\,u^{q-1}\cong -\big({u^{p-q}-\frac{s_1}{s_2}}\big)<0$, since $u\geq 1>\frac{s_1}{s_2}$. Thus $\varphi(u)\leq \varphi(1)=-q+p\,\frac{s_1}{s_2}-({p-q})\,s_1$,  $\forall\,u\in \big[{1,\frac{p}{p-1}}\big)$. 

We now prove that $p\,\frac{s_1}{s_2}<q+({p-q})\,s_1$, for every $(s_1,s_2)$ that satisfies $0<s_1^{q-1}\leq s_2^{p-1}<1$. For this proof we just need to justify  that 
\begin{align}
p\,\frac{s_1}{s_1^{\frac{q-1}{p-1}}}<q+({p-q})\,s_1\,,\quad\forall\,s_1\in({0,1})\,,\label{eq:2p15}
\end{align}
or equivalently $p\,s_1^{\frac{p-q}{p-1}}-({p-q})\,s_1-q<0\,,\;\forall\,s_1\in({0,1})$. Define $g({s_1}):=p\,s_1^{\frac{p-q}{p-1}}-({p-q})\,s_1-q\,,\;\forall\,s_1\in({0,1})$. Then $g'({s_1})=(p-q)\,\frac{p}{p-1}\,\frac{1}{s_1^{\frac{q-1}{p-1}}}-({p-q})\cong \frac{p}{p-1}\,\frac{1}{s_1^{\frac{q-1}{p-1}}}-1>\frac{p}{p-1}-1=\frac{1}{p-1}>0$. Thus $\forall\,s_1\in({0,1})$, $g(s_1)<g(1)=p-(p-q)-q=0$, and the claim is proved. 

We prove now the following 

\begin{lemma}\label{lem:2p1}
Let $(s_1,s_2)$ satisfying $s_2=s'_2=H_q({\omega_p(s_1)})$. Then if we define $\tau'=\tau({s_1,s'_2})=\frac{p-q}{p}\,\frac{t^p-s_1}{t^{p-q}-\frac{s_1}{s'_2}}$, then we have that $\tau'=H_q({\omega_p(s_1)})$.
\end{lemma}
\begin{proof}
We obviously have that $\omega_p(s_1)=\omega_q(s_2)$. Thus since then $t(s_1,s_2)$ equals $\omega_p(s_1)$ (by the results in \cite{13} - Lemma 2.3) we should have that the following equalities are true 
\begin{align}
\tau'=\frac{p-q}{p}\,\frac{\omega_p(s_1)^p-s_1}{\omega_p(s_1)^{p-q}-\frac{s_1}{H_q({\omega_p(s_1)})}}=\frac{p-q}{p}\,\frac{\lambda^p-H_p({\lambda})}{\lambda^{p-q}-\frac{H_p(\lambda)}{H_q({\lambda})}}\,,\label{eq:2p16}
\end{align}
where we have set $\lambda=\omega_p(s_1)\in\big[{1,\frac{p}{p-1}}\big)\;\Leftrightarrow\; s_1=H_p({\lambda})$. We just need to prove that 
\begin{align*}
\tau'&=H_q({\lambda})&&\stackrel{\text{by}\,\eqref{eq:2p16}}{\Longleftarrow\!\Longrightarrow}\vsal
\frac{p-q}{p}\,\frac{\lambda^p-H_p({\lambda})}{\lambda^{p-q}-\frac{H_p(\lambda)}{H_q({\lambda})}}&=H_q({\lambda})&&\Longleftrightarrow\vsal
\frac{p-q}{p}\,\big({\lambda^p-H_p({\lambda})}\big)&=\lambda^{p-q}H_q({\lambda})-H_p({\lambda})&&\Longleftrightarrow\vsal
\frac{p-q}{p}\,\big({\lambda^p-p\lambda^{p-1}+(p-1)\lambda^p}\big)&=\lambda^{p-q}\big({q\lambda^{q-1}-({q-1})\lambda^q}\big)\,-\vsal
&\hspace{1.5cm}-\big({p\lambda^{p-1}-(p-1)\lambda^{p}}\big)&&\Longleftrightarrow
\end{align*}
\begin{align*}
\frac{p-q}{p}\,\big({p\lambda^p-p\lambda^{p-1}}\big)&=q\lambda^{p-1}-({q-1})\lambda^p-p\lambda^{p-1}+(p-1)\lambda^p&&\Longleftrightarrow\vsal
\frac{p-q}{p}\,p\lambda^{p-1}(\lambda-1)&=(p-q)\lambda^{p}-(p-q)\lambda^{p-1}\,,
\end{align*}
which is obviously true. Thus Lemma \ref{lem:2p1} is proved.
\end{proof}

We now prove the following:

\begin{lemma}\label{lem:2p2}
Let $(s_1,s_2)$ satisfy $s_2=H_q({\omega_p(s_1)})=:s'_2$. Then 
\begin{align*}
F({\tau'})=(p-q)\,s'_2\bigg({\frac{\omega_q(s'_2)}{(q-1)\big({\omega_q(s'_2)-1}\big)}+\frac{\omega_q(s'_2)^q}{s'_2}}\bigg)
\end{align*}
and thus for this choice of $(s_1,s_2)$ we get that the right hand side of \eqref{eq:2p10} vanishes.
\end{lemma}
\begin{proof}
We have 
\begin{align*}
F({\tau'})&\stackrel{by\,\eqref{eq:2p9}}{=\!=\!=\!=}\frac{\tau'}{1-\omega_q({\tau'})}\bigg({p-\frac{q(p-1)}{q-1}\,\omega_q({\tau'})}\bigg)\,-\vsal
&\hspace{5.0cm}-q\Big({p\,\omega_q({\tau'})^{q-1}-({p-1})\,\omega_q({\tau'})^q}\Big)\,.
\end{align*}
Since by Lemma \ref{lem:2p1} $\omega_q({\tau'})=\omega_p({s_1})=:\lambda$, for the proof of the present Lemma we just need to show that:
\begin{align}
\frac{H_q({\lambda})}{1-\lambda}\bigg({p-\frac{q(p-1)}{q-1}\,\lambda}\bigg)&-q\big({p\lambda^{q-1}-(p-1)\lambda^q}\big)=\nonumber\vsal
&(p-q)\,H_q(\lambda)\bigg({\frac{\lambda}{(q-1)(\lambda-1)}+\frac{\lambda^q}{H_q(\lambda)}}\bigg)\,.\label{eq:2p17}
\end{align}
Then
\begin{align}
&\hspace{9.5cm}\eqref{eq:2p17} &&\Leftrightarrow\nonumber\vsal
&\frac{q\lambda^{q-1}-(q-1)\lambda^q}{1-\lambda}\bigg({p-\frac{q(p-1)}{q-1}\,\lambda}\bigg)-q\big({p\lambda^{q-1}-(p-1)\lambda^q}\big)=\nonumber\vsal
&\hspace{1,5cm}(p-q)\big({q\lambda^{q-1}-(q-1)\lambda^q}\big)\bigg({\frac{\lambda}{(q-1)(\lambda-1)}+\frac{\lambda}{q-(q-1)\lambda}}\bigg)&&\Leftrightarrow\nonumber\vsal
&\frac{q-(q-1)\lambda}{1-\lambda}
\bigg({p-\frac{q(p-1)}{q-1}\,\lambda}\bigg)-q\big({p-(p-1)\lambda}\big)= \nonumber\vsal
&\hspace{2.0cm}(p-q)\big({q-(q-1)\lambda}\big)\bigg({\frac{1}{(q-1)(\lambda-1)}+\frac{1}{q-(q-1)\lambda}}\bigg)\lambda\,.\label{eq:2p18}
\end{align}
Now \eqref{eq:2p18} is easily seen to be true, after doing some simple calculations. Lemma \ref{lem:2p2} is now proved.
\end{proof}
Finally we prove the following 

\begin{lemma}\label{lem:2p3}
Define the function $A$, by 
\begin{align*}
	A({s_2})=s_2\bigg({\frac{\omega_q(s_2)}{(q-1)\big({\omega_q(s_2)-1}\big)}+\frac{\omega_q(s_2)^q}{s_2}}\bigg)\,,\quad s_2\in({0,1})\,.
\end{align*}
Then $A({s_2})$ is strictly increasing.
\end{lemma}
\begin{proof}
For the proof we define $\omega_q(s_2)=\gamma\in\big({1,\frac{q}{q-1}}\big)\;\Leftrightarrow\; s_2=H_q(\gamma)$. We need to prove that the function 
\begin{align*}
	A_1({\gamma})=H_q(\gamma)\bigg({\frac{\gamma}{(q-1)(\gamma-1)}+\frac{\gamma^q}{H_q(\gamma)}}\bigg)\,,
\end{align*}
is strictly decreasing. This is enough, since $H_q(s_2)$ is strictly decreasing function of $s_2\in({0,1})$. We have 
\begin{align*}
A_1({\gamma})&=H_q(\gamma)\,\frac{\gamma}{(q-1)(\gamma-1)}+\gamma^q\hspace{2.0cm}\Rightarrow\vsal
A'_1({\gamma})&=q\gamma^{q-1}+q(q-1)\gamma^{q-2}({1-\gamma})\,\frac{\gamma}{(q-1)(\gamma-1)}\,+\vsal
&\hspace{1.0cm}+\big({q\gamma^{q-1}-({q-1})\gamma^q}\big)\,\frac{1}{q-1}\Big({-\frac{1}{({\gamma-1})^2}}\Big)<0\,,\quad \forall\,\gamma\in\big({1,\tfrac{q}{q-1}}\big)\,.
\end{align*}
Lemma \ref{lem:2p3}, is now proved.
\end{proof}

We proceed to prove \eqref{eq:2p1}.

\noindent i) Suppose that $s_2\in({s'_2,1})$ where $s'_2=H_q({\omega_p(s_1)})$. We prove that $\frac{\partial t}{\partial s_2}(s_1,s_2)<0$, for each such $s_2$. It is enough to prove, because of  \eqref{eq:2p10}, that $\forall\,s_2\in({s'_2,1})$, \; $F(\tau)<(p-q)\,A(s_2)$, where $A(s_2)$ is defined in Lemma \ref{lem:2p3}. 

At this part we shall need the following:

\begin{lemma}\label{lem:2p4}
The function $F(\tau)$, (defined above by \eqref{eq:2p9}) 
\begin{align*}
	F({\tau})&=\frac{\tau}{1-\omega_q(\tau)}\bigg({p-\frac{q\,(p-1)}{q-1}\,\omega_q(\tau)}\bigg)-q\Big({p\,\omega_q(\tau)^{q-1}-({p-1})\,\omega_q(\tau)^{q}}\Big)
\end{align*}
is strictly increasing for $\tau\in(0,1)$\,.
\end{lemma}
\begin{proof}
It is enough (since $H_q(\gamma)$ is strictly decreasing function for $\gamma\in\big({1,\tfrac{q}{q-1}}\big)$), to prove that the function ( we set $\omega_q(\tau)=\gamma\;\Leftrightarrow\; \tau=H_q(\gamma)$ )
\begin{align*}
	F_1({\gamma})&=\frac{H_q(\gamma)}{1-\gamma}\bigg({p-\frac{q\,(p-1)}{q-1}\,\gamma}\bigg)-q\big({p\gamma^{q-1}-({p-1})\gamma^{q}}\big)
\end{align*}
is strictly decreasing for $\gamma\in\big({1,\tfrac{q}{q-1}}\big)$. For this purpose we calculate  $H_q'(\gamma)=q(q-1)\,\gamma^{q-2}(1-\gamma)$ and 
\begin{align*}
\Bigg({\frac{\frac{q\,(p-1)}{q-1}\,\gamma-p}{\gamma-1}}\Bigg)'_{\gamma}&=\frac{\frac{q\,(p-1)}{q-1}({\gamma-1})-\big({\frac{q\,(p-1)}{q-1}\,\gamma-p}\big)}{({\gamma-1})^2}\vsal
 &=\frac{1}{q-1}\frac{1}{({\gamma-1})^2}\big({q(p-1)\gamma-q(p-1)-q(p-1)\gamma+p(q-1)}\big)\vsal
&=-\frac{p-q}{q-1}\,\frac{1}{({\gamma-1})^2}<0\,,\qquad \forall\,\gamma\in\big({1,\tfrac{q}{q-1}}\big)\,.
\end{align*}
Lemma \ref{lem:2p4}, is now proved.
\end{proof}

Remember now that $\tau=\tau({s_1,s_2,t})=\frac{p-q}{p}\,\frac{t^p-s_1}{t^{p-q}-\frac{s_1}{s_2}}$, is strictly increasing function for $t\in[1,+\infty)$, as can be easily seen by differentiating with respect to $t$. Thus since $s'_2<s_2<1$, we get the following. Of course we remind that $t(s_1,s_2)\leq \omega_p(s_1)=t(s_2,s'_2)$ where the equality  just written is due to the fact that $s'_2=H_q({\omega_p(s_1)})$. Thus we have:
\begin{align*}
\tau=\tau({s_1,s_2,t})&=\frac{p-q}{p}\,\frac{t^p-s_1}{t^{p-q}-\frac{s_1}{s_2}}\leq \frac{p-q}{p}\,\frac{\omega_p(s_1)^p-s_1}{\omega_p(s_1)^{p-q}-\frac{s_1}{s_2}}\vsal
& < \frac{p-q}{p}\,\frac{t({s_1,s'_2})^p-s_1}{t({s_1,s'_2})^{p-q}-\frac{s_1}{s'_2}}\vsal
&=\tau({s_1,s'_2})=H_q\big({\omega_p(s_1)}\big)\,,
\end{align*}
where the last equality is due to Lemma \ref{lem:2p1}. Then since 
\begin{align*}
\tau<\tau({s_1,s'_2})&=\tau'=H_q\big({\omega_p(s_1)}\big)\qquad\stackrel{ \text{by\,Lemma}\,\ref{lem:2p4}}{=\!=\!=\!=\!=\!=\!\Longrightarrow}\vsal
F({\tau})&<F({\tau'})=(p-q)\,A(s'_2)\,,
\end{align*}
where  the equality above is justified by Lemma \ref{lem:2p2}. But $s'_2<s_2<1$, and by  Lemma \ref{lem:2p3}, $A(s_2)$ is strictly increasing function for $s_2\in({0,1})$. Thus we get for every $s_2\in({s'_2,1})$ that: $F({\tau})<(p-q)\,A(s_2)$ which is the desired result. Thus $\frac{\partial \,t}{\partial s_2}<0$,\;  $\forall\,s_2\in({s'_2,1})$.

\medskip 

\noindent ii) Suppose now that $s_2\in({s''_2,s'_2})$, where $s''_2:=\max\Big({s_1^{\frac{q-1}{p-1}},h^{-1}\big({\frac{1}{s_1}}\big)}\Big)$, \; $s'_2=H_q({\omega_p(s_1)})$. We prove that for each such $s_2$, we have $\frac{\partial \,t}{\partial s_2}({s_1,s_2})>0$. Because of \eqref{eq:2p10} it is enough to prove that $s_2\in({s''_2,s'_2})$ implies 
\begin{align}
F({\tau})>(p-q)\,A({s_2})=({p-q})s_2\Bigg({\frac{\omega_q({s_2})}{({q-1})\,\big({\omega_q({s_2})-1}\big)}+\frac{\omega_q({s_2})^q}{s_2}}\Bigg)\,.\label{eq:2p19}
\end{align}
Since $s''_2<s_2<s'_2$ we have, because of Lemmas \ref{lem:2p3}. and \ref{lem:2p2}:
\begin{align}
(p-q)\,A(s_2)<(p-q)\,A(s'_2)=F({\tau'})\,,\label{eq:2p20}
\end{align}
where $\tau'=H_q({\omega_p(s_1)})=s'_2$, by Lemma \ref{lem:2p1}. 

Thus if we prove that $\tau'<\tau$, then  \eqref{eq:2p20} and Lemma \ref{lem:2p4} will imply $(p-q)\,A(s_2)<F({\tau'})<F({\tau})$, which is \eqref{eq:2p19}, thus we get the desired result. We proceed to prove that 
\begin{align}
&\hspace{2.3cm}\tau'<\tau\,.\label{eq:2p21}\vsal
\eqref{eq:2p21}\quad&\Leftrightarrow\quad H_q({\omega_p(s_1)})<\tau=\tau({s_1,s_2,t})\nonumber\vsal
&\Leftrightarrow\quad\omega_p(s_1)>\omega_q(\tau)\,.\label{eq:2p22}
\end{align}
Define the function $\varphi({x})=p\,x^{q-1}-(p-1)\,x^q$\,,\; $x>1$. Then $\varphi'({x})\cong p\,({q-1})-(p-1)\,q\,x<-p+q<0$. Remember also that $s_1$, $s_2$, $t$, $\tau$ satisfy (by the definition of $t=t(s_1,s_2)$)
\begin{align*}
q\Big({p\,\omega_q(\tau)^{q-1}-({p-1})\,\omega_q(\tau)^{q}}\Big)=\frac{({p-q})s_1\,\alpha(s_2)}{t^{p-q}-\frac{s_1}{s_2}}\,.
\end{align*}
Thus by the above comments 
\begin{align}
\eqref{eq:2p22}\quad&\Leftrightarrow\quad q\,\varphi\big({\omega_p(s_1)}\big)< q\,\varphi\big({\omega_q(\tau)}\big)\nonumber\vsal
&\Leftrightarrow\quad \frac{({p-q})s_1\,\alpha(s_2)}{t^{p-q}-\frac{s_1}{s_2}}>q\Big({p\,\omega_p(s_1)^{q-1}-({p-1})\,\omega_p(s_1)^{q}}\Big)\nonumber\vsal
&\Leftrightarrow\quad \frac{({p-q})s_1\,\alpha(s_2)}{t^{p-q}-\frac{s_1}{s_2}}>\frac{q\,H_p\big({\omega_p(s_1)}\big)}{\omega_p(s_1)^{p-q}}=\frac{q\,s_1}{\omega_p(s_1)^{p-q}}\nonumber\vsal
&\Leftrightarrow\quad \frac{({p-q})\,\alpha(s_2)}{t^{p-q}-\frac{s_1}{s_2}}>\frac{q}{\omega_p(s_1)^{p-q}}\,.\label{eq:2p23}
\end{align}
Remember that  $\alpha({s_2})=\frac{\omega_q(s_2)^{q}}{s_2}-1$ and then since 
\begin{align*}
s_2<s'_2\quad\Rightarrow \quad \alpha({s_2})>\alpha({s'_2})&=\frac{\omega_q(s'_2)^{q}}{s'_2}-1=\frac{\omega_p(s_1)^{q}}{H_q\big({\omega_p(s_1)}\big)}-1\vsal
                          &=\frac{\omega_p(s_1)^{q}}{q\,\omega_p(s_1)^{q-1}-({q-1})\,\omega_p(s_1)^{q}}-1\vsal
                          &=\frac{\omega_p(s_1)}{q-(q-1)\,\omega_p(s_1)}-1\vsal
                          &=\frac{\omega_p(s_1)-q+(q-1)\,\omega_p(s_1)}{q-(q-1)\,\omega_p(s_1)}\vsal
                          &=\frac{q\big({\omega_p(s_1)-1}\big)}{q-(q-1)\,\omega_p(s_1)}\,.
\end{align*}
Thus in order to prove \eqref{eq:2p23}, we just need to prove that the following inequality is true:
\begin{align}
\frac{(p-q)\,\frac{q({\omega_p(s_1)-1})}{q-(q-1)\,\omega_p(s_1)}}{t^{p-q}-\frac{s_1}{s_2}}\geq \frac{q}{\omega_p(s_1)^{p-q}}\,.\label{eq:2p24}
\end{align}
But $t=t(s_1,s_2)\leq \omega_p(s_1)$, thus for the validity of \eqref{eq:2p24} it is enough to prove the inequality:
\begin{align}
(p-q)\,\frac{q\big({\omega_p(s_1)-1}\big)}{q-(q-1)\,\omega_p(s_1)}\geq q\bigg({\omega_p(s_1)^{p-q}-\frac{s_1}{s_2}}\bigg)\,\omega_p(s_1)^{-p+q}\,.\label{eq:2p25}
\end{align}
But $s_2<s'_2=H_q({\omega_p(s_1)})$, thus \eqref{eq:2p25} is true if 
\begin{align}
	(p-q)\,\frac{q\big({\omega_p(s_1)-1}\big)}{q-(q-1)\,\omega_p(s_1)}\geq q\bigg({\omega_p(s_1)^{p-q}-\frac{s_1}{H_q({\omega_p(s_1)})}}\bigg)\,\omega_p(s_1)^{-p+q}\,.\label{eq:2p26}
\end{align}
But as will immediately see, in fact have equality in  \eqref{eq:2p26}. Indeed, set $t=\omega_p(s_1)\in\big({1,\frac{p}{p-1}}\big)$. We need to prove, 
\begin{align*}
(p-q)\,\frac{t-1}{q-(q-1)\,t}&=\bigg({t^{p-q}-\frac{H_p(t)}{H_q(t)}}\bigg)\,t^{-p+q}&&\Leftrightarrow\vsal
(p-q)\,\frac{t-1}{q-(q-1)\,t}&=1-\frac{1}{t^{p-q}}\,\frac{p\,t^{p-1}-(p-1)\,t^p}{q\,t^{q-1}-(q-1)\,t^q}&&\Leftrightarrow\vsal
(p-q)\,\frac{t-1}{q-(q-1)\,t}&=1-\frac{1}{t^{p-q}}\,\frac{p\,t^{p-q}-(p-1)\,t^{p-q+1}}{q-(q-1)\,t}&&\Leftrightarrow\vsal
(p-q)(t-1)&=q-(q-1)\,t-p+(p-1)\,t&&\Leftrightarrow\vsal
(p-q)(t-1)&=-(p-q)+(p-q)\,t\,,
\end{align*}
which is obviously true. Relations \eqref{eq:2p1} are now proved. \qed

\vspace{50pt}
\noindent Nikolidakis Eleftherios\\
Assistant Professor\\
Department of Mathematics \\
Panepistimioupolis, University of Ioannina, 45110\\
Greece\\
E-mail address: enikolid@uoi.gr

\end{document}